\documentclass[12pt]{amsart}

\usepackage{amssymb}
\usepackage{amsmath}
\usepackage{amsthm}
\usepackage{graphicx}
\usepackage{enumerate}
\usepackage{tikz}
\usepackage{tabularx}
\usepackage{tabmac_avi}
\usetikzlibrary{arrows.meta}
\tikzset{>={Latex[width=1.2mm,length=1.7mm]}}

\newtheorem{thm}{Theorem}[section]
\newtheorem{prop}[thm]{Proposition}

\newtheorem{lem}[thm]{Lemma}

\newtheorem{prob}[thm]{Problem}

\numberwithin{equation}{section}

\newcommand{\sn}{\mathfrak{S}_n}

\newcommand{\mfs}[1]{\mathfrak{S}_{#1}}

\newcommand{\zsn}{\mathbb{Z}[\sn]}

\newcommand{\cx}{\mathbb{C}[x]}

\newcommand{\csn}{\mathbb{C}[\sn]}

\newcommand{\ol}[1]{\overline{#1}}

\newcommand{\K}{\mathcal{B}}

\newcommand{\imm}[1]{\mathrm{Imm}_{#1}}

\newcommand{\sumsb}[1]{\sum_{\substack{#1}}} 

\newcommand{\inv}{\textsc{inv}}

\newcommand{\defeq}{:=}

\newcommand{\spn}{\mathrm{span}}
\newcommand{\hgt}{\mathrm{hgt}}
\newcommand{\sgn}{\mathrm{sgn}}
\newcommand{\triv}{\mathrm{triv}}

\newcommand{\tr}{{\negthickspace \top \negthickspace}}

\newcommand{\ntnsp}{\negthinspace}
\newcommand{\ntksp}{\negthickspace}
\newcommand{\nTksp}{\negthickspace\negthickspace}

\newcommand{\nTtksp}{\negthickspace\negthickspace\negthickspace}

\newcommand{\bp}{\begin{prob}}
\newcommand{\ep}{\end{prob}}

\newcommand{\mat}[1]{\mathrm{Mat}_{#1 \times #1}}

\newcommand{\perm}{\mathrm{per}}
\newcommand{\ctype}{\mathrm{ctype}}

\newcommand{\permmon}[2]{#1_{1,#2_1} \ntnsp\cdots {#1}_{n,#2_n}}

\newcommand{\ssm}{\smallsetminus}

\newcommand{\trspace}[1]{\mathcal{T}_{#1}}
\newcommand{\upparrow}{\big \uparrow \nTksp \phantom{\uparrow}}

\newcommand{\tn}{T_n(\xi)}

\begin{document}
\author{Mark Skandera}
\author{Daniel Soskin}
\title{Barrett--Johnson inequalities for totally nonnegative matrices}

\bibliographystyle{dart}

\date{\today}

\begin{abstract}
Given a matrix $A$, let $A_{I,J}$ denote the submatrix of $A$ determined by
rows $I$ and columns $J$.  Fischer's Inequalities state that for each $n \times n$ Hermitian positive semidefinite
matrix $A$, and each subset
$I$ of $\{1,\dotsc,n\}$ and its complement $I^c$, we have
$\det(A) \leq \det(A_{I,I})\det(A_{I^c,I^c})$. Barrett and Johnson (Linear Multilinear Algebra 34, 1993)
extended these to state inequalities 
for sums of products of principal minors whose orders
are given by nonincreasing integer sequences $(\lambda_1,\dotsc,\lambda_r)$,
$(\mu_1,\dotsc,\mu_s)$ summing to $n$.  Specifically, if $\lambda_1+\cdots+\lambda_i\leq \mu_1+\cdots+\mu_i$ for all $i$, then
$$
\lambda_1!\cdots\lambda_r!
\nTksp\sum_{(I_1,\dotsc,I_r)}\nTksp
\det(A_{I_1,I_1}) \cdots \det(A_{I_r,I_r}) ~\geq~
  \mu_1!\cdots\mu_s!
  \nTksp\sum_{(J_1,\dotsc,J_s)}\nTksp
  \det(A_{J_1,J_1}) \cdots \det(A_{J_s,J_s}),
$$
where sums are over sequences of disjoint subsets
of $\{1,\dotsc,n\}$ satisfying $|I_k| = \lambda_k$,
$|J_k| = \mu_k$.
We show that these inequalities hold for totally nonnegative matrices
as well.
\end{abstract}

\maketitle
\section{Introduction}\label{s:intro}
A matrix $A \in \mat{n}(\mathbb C)$ 
is called \emph{Hermitian}
if it satisfies $A^*=A$ where $*$ denotes conjugate transpose. 
Such a matrix is called \emph{Hermitian positive semi-definite} (HPSD) if we have $x^* A x \geq 0$ for all $x \in \mathbb C^n.$ 
 For $A \in \mat n(\mathbb R)$, the Hermitian property reduces to symmetry $A^\tr = A$, and the positive semidefinite (PSD) property reduces to $x^\tr A x \geq 0$ for all $x \in \mathbb R^n$. Hermitian matrices are fundamental to quantum mechanics since they describe operators with real eigenvalues. Positive definite (semi-definite) matrices are one of the matrix analogs of positive (nonnegative) numbers, and correspond to inner products (nonnegative quadratic forms). They appear in many fields such as geometry, numerical analysis, optimization, quantum physics, and statistics. See \cite{Bhatia} and \cite{MatAnHJ}.

Another matrix analog of nonnegative numbers is the
class of \emph{totally nonnegative} (TNN) matrices, 
those matrices in
$\mat{n}(\mathbb R)$
in which each minor (determinant of a square submatrix) is
nonnegative. Initially, total nonnegativity arose in three different areas. It was studied by Gantmacher and Krein in oscillations of vibrating systems, by Schoenberg in applications to analysis of real roots of polynomials and spline functions, and by Karlin in integral equations and statistics. TNN matrices appear in many other areas of mathematics. (See, e.g., \cite{fallat2011total}.) 


Like TNN matrices, HPSD matrices have a well-known characterization in terms of matrix minors.  
Given an $n \times n$ matrix $A = (a_{i,j})$
and subsets $I,J \subseteq [n] \defeq \{1,\dotsc,n\}$,
define the submatrix $A_{I,J} = (a_{i,j})_{i \in I, j \in J}$, and define the set $I^c \defeq [n] \ssm I$.
While TNN matrices are characterized by the inequalities
\begin{equation}\label{eq:tnn}
\det(A_{I,J}) \geq 0, \qquad \text{for all  $I, J  \subseteq [n]$ with } |I|=|J|,
\end{equation}
HPSD matrices are characterized by the equalities
\begin{equation}\label{eq:hpsd}
\begin{aligned}
a_{j,i}^*  &= a_{i,j}, &\qquad &\text{for all $i,j \in [n]$},\\   
\det(A_{I,I}) &\geq 0, &\qquad &\text{for all $I \subseteq [n]$}. 
\end{aligned}
\end{equation}

From the inequalities in (\ref{eq:tnn}) -- (\ref{eq:hpsd}), one can deduce many 
inequalities satisfied by polynomials in the entries of $n \times n$ HPSD, PSD, and/or TNN matrices.  
The following inequalities hold for  
HPSD matrices and TNN matrices.
Hadamard~\cite{hadamard1893} showed that for $A$ HPSD we have
\begin{equation}\label{eq:hadamard}
    \det(A)\leq a_{1,1} \cdots a_{n,n},
    \end{equation}
and     Koteljanskii~\cite{KotelProp}, \cite{KotelPropRuss} showed that this holds for $A$ TNN as well.
    Marcus~\cite{marcus1963perm} proved a permanental analog
    \begin{equation}\label{eq:marcus}
    \perm(A)\geq a_{1,1} \cdots a_{n,n}
    \end{equation}
    for $A$ HPSD, and this analog clearly holds for $A$ TNN as well.
    Fischer~\cite{fischer1908} strengthened (\ref{eq:hadamard}) by showing that for all $I \subseteq [n]$ we have
    \begin{equation}\label{eq:fischer}
    \det(A)\leq \det(A_{I,I}) \det(A_{I^c,I^c}),
    \end{equation}
     and Ky Fan showed that this holds for $A$ TNN as well (unpublished; see \cite{carlson1967}).
    Lieb~\cite{LiebPerm} proved a permanental analog
    \begin{equation}\label{eq:lieb}
    \perm(A)\geq \perm(A_{I,I}) \;\perm(A_{I^c,I^c}),
    \end{equation}
    for $A$ HPSD, and this analog holds for $A$ TNN because every term in the expansion of the product on the right-hand side of (\ref{eq:lieb}) also appears on the left-hand side.
    Koteljanskii~\cite{KotelProp}, \cite{KotelPropRuss} strengthened (\ref{eq:fischer}) further by proving that for all $I, J \subseteq [n]$ we have
    \begin{equation}\label{eq:koteljanskii}
    \det(A_{I\cup J, I\cup J})\det(A_{I\cap J, I\cap J})\leq \det(A_{I,I}) \det(A_{J,J})
    \end{equation}
    for $A$ belonging to a class of matrices including HPSD and TNN matrices.  
    Strengthening (\ref{eq:hadamard}) somewhat differently, Schur~\cite{SchurUberendliche} proved that
    for each $\sn$-character $\chi$, the inequality
    \begin{equation}\label{eq:schur}
    \det(A) \leq \frac{\imm{\chi}(A)}{\chi(e)}
    \end{equation}
    holds for $A$ HPSD, where $\imm{\chi}$ is the corresponding immanant (defined in \S \ref{s:imm}).  Stembridge~\cite[Cor.~3.4]{StemImm} showed that this holds for $A$ TNN as well.
    
    
Other inequalities hold for 
TNN matrices, and are not known to hold for HPSD matrices.
For instance Charles Johnson showed that the permanental analog of (\ref{eq:schur}), 
\begin{equation}\label{eq:permdominance}
\perm(A) \geq \frac{\imm{\chi}(A)}{\chi(e)},
\end{equation}
holds for TNN matrices (unpublished; see \cite[p.\,1088]{StemConj}).  The validity of this inequality for HPSD matrices was conjectured by Lieb~\cite{LiebPerm}, is still open, and is known as 
the {\em permanental dominance conjecture}.
Extending the results (\ref{eq:fischer}) and (\ref{eq:koteljanskii}), Fallat, Gekhtman, and Johnson~\cite{FGJMult}
and the first author~\cite{SkanIneq} characterized all $8$-tuples
$(I,J,K,L,I',J',K',L')$ of subsets of $[n]$
for which we have the inequality
\begin{equation}\label{eq:fgjskan}
\det(A_{I,I'})\det(A_{J,J'}) \leq 
\det(A_{K,K'})\det(A_{L,L'})
\end{equation}
whenever $A$ is TNN.
(See \cite{GkhtSos2022bounded} for progress on products of three or more minors.) 
Drake, Gerrish, and the first author~\cite{DGSBruhat} proved a similar result for products of $n$ matrix entries.
In particular, the inequality
\begin{equation}\label{eq:dgs}
    \permmon av \geq \permmon aw
\end{equation}
holds for all TNN $A$ if and only if $v \leq w$ in the Bruhat order on $\sn$. (See, e.g., \cite{BBCoxeter}.)
 Other linear combinations of the monomials in (\ref{eq:dgs}) which produce valid inequalities for TNN matrices are related to Lusztig's work on canonical bases of quantum groups~\cite{LusztigTPCB} and are called \emph{Kazhdan-Lusztig immanants} in \cite{RSkanKLImm}.
 Writing these as
 $\{ \imm{w}(A) \,|\, w \in \sn \}$, we have
 \begin{equation}\label{eq:klimm}
 \imm w(A) \geq 0
 \end{equation}
for all $w \in \sn$ and all TNN $A$.
(See \cite{RSkanKLImm} for definitions.)
We will define and use special cases of these,
called {\em Temperley-Lieb immanants} in \cite{RSkanTLImmp}, in Section~\ref{s:tl2c}.

Still other inequalities hold for HPSD matrices and are not known to hold for TNN matrices.
Heyfron~\cite{HeyfronImmDom} extended (\ref{eq:schur}) and proved special cases of the permanental dominance conjecture (\ref{eq:permdominance}) by considering irreducible $\sn$-characters
indexed by "hook" shapes $\lambda = (k, 1, \dotsc, 1)$.  (See Sections \ref{s:sn} -- \ref{s:imm} for definitions.)
In particular, we have
\begin{equation}\label{eq:heyfron}
   \perm(A)=\frac{\imm{\chi^n}(A)}{\chi^{n}(e)}\geq \frac{\imm{\chi^{n-1,1}}(A)}{\chi^{n-1,1}(e)}\geq \frac{\imm{\chi^{ n-2,1,1}}(A)}{\chi^{n-2,1,1}(e)}\geq \cdots \geq \frac{\imm{\chi^{1,\cdots,1}}(A)}{\chi^{1,\cdots,1}(e)}=\det(A)
\end{equation} 
for $A$ HPSD.  (See \cite{pate1999} for more results of this form.)
    Borcea and Br\"and\'en showed that averages
    of the products of pairs of minors appearing in (\ref{eq:fischer}) satisfy the log-concavity inequalities~\cite[Cor.\,3.1\,(b)]{BBApps} 
    \begin{equation}\label{eq:borceabranden1}
        \left( \sumsb{I \subseteq [n]\\|I|=k} \frac{\det(A_{I,I}) \det(A_{I^c,I^c})}{\displaystyle{\binom nk}} \ntksp\right)^{\nTksp 2}
    \geq 
    \left(\sumsb{I \subseteq[n]\\|I|=k+1} \nTksp\frac{\det(A_{I,I}) \det(A_{I^c,I^c})}{\displaystyle{\binom n{k+1}}}\ntksp\right)
    \ntksp \left(\sumsb{I \subseteq [n]\\|I|=k-1}\nTksp\frac{ \det(A_{I,I}) \det(A_{I^c,I^c})}{\displaystyle{\binom n{k-1}}}\ntksp\right)
    \end{equation}
    for all $A$ HPSD and $k = 1,\dotsc,n-1$.
    They also showed that these averages
    satisfy the Maclaurin-type inequalities~\cite[Cor.\,3.1\,(c)]{BBApps} 
    \begin{equation}\label{eq:borceabranden2}
        \left(\frac{\displaystyle{\sum}_{|I|=k+1} \det(A_{I,I}) \det(A_{I^c,I^c})}{\displaystyle{\binom n{k+1}}\det(A)}\right)^{\ntksp k} \ntnsp \geq \left(\frac{\displaystyle{\sum}_{|I|=k} \det(A_{I,I}) \det(A_{I^c,I^c})}{\displaystyle{\binom n{k}}\det(A)}\right)^{k+1}
    \end{equation}
    for all nonsingular $A$ HPSD and $k = 1,\dotsc,n-1$. 

Nearly belonging to the HPSD list above are 
the Barrett--Johnson Inequalities~\cite{BJMajor} for (real) PSD matrices.  Given nonnegative integer sequences
$\lambda = (\lambda_1,\dotsc,\lambda_r)$, $\mu = (\mu_1,\dotsc,\mu_s)$ summing to $n$, we have
\begin{equation}\label{eq:bj}
\lambda_1!\cdots\lambda_r!
\nTksp\sum_{(I_1,\dotsc,I_r)}\nTksp
\det(A_{I_1,I_1}) \cdots \det(A_{I_r,I_r}) ~\geq~
  \mu_1!\cdots\mu_s!
  \nTksp\sum_{(J_1,\dotsc,J_s)}\nTksp
  \det(A_{J_1,J_1}) \cdots \det(A_{J_s,J_s}),
\end{equation}
if and only if $\lambda_1 + \cdots + \lambda_i \leq \mu_1 + \cdots + \mu_i$ for all $i$,
where the sums are over
sequences of disjoint subsets of $[n]$ having cardinalities $\lambda$, $\mu$, respectively.
We will show that these inequalities also hold for TNN matrices.

In Section~\ref{s:sn}, we review basic facts about the symmetric group, its traces (class functions), and symmetric functions. 
In Section~\ref{s:imm} we apply these ideas to form certain polynomials called {\em immanants} in the entries of $n \times n$ matrices, and we discuss the evaluation of these at TNN matrices. 
In Section~\ref{s:tl2c}, we discuss the Temperley--Lieb algebra, related immanants,
and their relationship to products of matrix minors.
In Section~\ref{s:mn}, we present our main result
that the Barrett--Johnson inequalities hold for
all $n \times n$ TNN matrices. 

\section{The symmetric group, its traces, and symmetric functions}\label{s:sn}

The {\em symmetric group algebra} $\csn$ is
generated over $\mathbb C$ by $s_1,\dotsc, s_{n-1}$, subject to relations
\begin{equation*}
\begin{alignedat}{2}
s_i^2 &= e &\qquad
&\text{for $i = 1, \dotsc, n-1$},\\
s_is_js_i &= s_js_is_j &\qquad
&\text{for $|i - j| = 1$},\\
s_is_j &= s_js_i &\qquad
&\text{for $|i - j| \geq 2$}.
\end{alignedat}
\end{equation*}
We define the {\em one-line notation} $w_1 \cdots w_n$ of $w \in \sn$ by
letting any expression for $w$ act on the word $1 \cdots n$,
where each generator $s_j = s_{[j,j+1]}$ acts on an $n$-letter word by
swapping the letters in positions $j$ and $j+1$, i.e.,
$s_j \circ v_1 \cdots v_n = v_1 \cdots v_{j-1} v_{j+1} v_j v_{j+2} \cdots v_n$. Whenever $s_{i_1}\cdots s_{i_\ell}$ is a reduced (short as possible) expression for $w \in \sn$, we call $\ell$ the {\em length}
of $w$ and write $\ell = \ell(w)$.  It is known that $\ell(w)$ is equal
to $\inv(w)$, the number of inversions 
in 
$w_1 \cdots w_n$. 
It is also known that conjugacy classes in $\sn$ are precisely the set of permutations that have the same cycle type. We name cycle type by integer
partitions of $n$, the
weakly decreasing positive integer sequences
$\lambda = (\lambda_1, \dotsc, \lambda_\ell)$ satisfying
$\lambda_1 + \cdots + \lambda_\ell = n$.
The $\ell = \ell(\lambda)$ components of $\lambda$ are called its {\em parts},
and we let $|\lambda| = n$ and $\lambda \vdash n$ denote that
$\lambda$ is a partition of $n$.  
Given $\lambda \vdash n$,
we define the {\em transpose} partition
$\lambda^\tr = (\lambda^\tr_1, \dotsc, \lambda^\tr_{\lambda_1})$
by
\begin{equation*}
  \lambda^\tr_i = \# \{ j \,|\, \lambda_j \geq i \}.
\end{equation*}
Sometimes it is convenient to name a partition with exponential notation,
omitting parentheses and commas, so that
$4^21^6 \defeq (4, 4, 1, 1, 1, 1, 1, 1)$. Sometimes it is convenient to partially order partitions of $n$ by the \emph{majorization order}, \begin{equation}
    \lambda\preceq\mu \quad \text{iff} \quad \lambda_{1}+\cdots+\lambda_{i}\leq \mu_{1}+\cdots+\mu_{i} 
\end{equation} 
for all $i$. It is known that if $\lambda$ is covered by $\mu$ in this partial order then $\lambda$, $\mu$ have equal parts except in two positions $i < j$ where we have
\begin{equation}\label{eq:movebox}
    \mu_i=\lambda_i+1, \qquad \mu_j=\lambda_j-1.
\end{equation}

Let $\Lambda$ be the ring of symmetric functions in 
$x = (x_1, x_2, \dotsc)$ having integer coefficients, and
let $\Lambda_n$ be the $\mathbb Z$-submodule of homogeneous functions
of degree $n$.
This submodule has rank equal to the number of partitions of $n$.
Five standard bases of $\Lambda_n$ consist of the monomial $\{m_\lambda \,|\, \lambda \vdash n\}$,
elementary $\{e_\lambda \,|\, \lambda \vdash n\}$,
(complete) homogenous $\{h_\lambda \,|\, \lambda \vdash n\}$,
power sum $\{p_\lambda \,|\, \lambda \vdash n\}$,
and Schur $\{s_\lambda \,|\, \lambda \vdash n\}$ 
symmetric functions.
(See, e.g., \cite[Ch.\,7]{StanEC2} for definitions.)

The change-of-basis matrix relating
$\{ e_\lambda \,|\, \lambda \vdash n \}$ and
$\{ m_\lambda \,|\, \lambda \vdash n \}$ is given by the equations
\begin{equation}\label{eq:etom}
  e_\lambda = \sum_{\mu \preceq \lambda^\tr} M_{\lambda,\mu} m_\mu,
\end{equation}
where $M_{\lambda,\mu}$ equals the number of column-strict Young tableaux
of shape $\lambda^\tr$ and content $\mu$. That is, $M_{\lambda,\mu}$ is the number of histograms having $\lambda_{i}$ boxes in column $i$ for all $i$, filled with $\mu_1$ ones, $\mu_2$ twos, etc., with column entries strictly increasing from bottom to top.
For example we have 
\begin{equation*}
e_{221}=m_{32}+2m_{311}+5m_{221}+12m_{2111}+30m_{11111},     
\end{equation*}
and the coefficients $M_{221,32}=1$ of $m_{32}$,
$M_{221,311}=2$ of $m_{311}$,
$M_{221,221}=5$ of $m_{221}$
count the column-strict tableaux
\begin{equation*}
\tableau[scY]{2,2|1,1,1}\,,\qquad\qquad\qquad
\tableau[scY]{2,3|1,1,1}\,,\quad
\tableau[scY]{3,2|1,1,1}\,,
\end{equation*}
\begin{equation*}
\tableau[scY]{2,2|1,1,3}\,,\quad
\tableau[scY]{2,3|1,1,2}\,,\quad
\tableau[scY]{3,2|1,1,2}\,,\quad
\tableau[scY]{2,3|1,2,1}\,,\quad
\tableau[scY]{3,2|2,1,1}\,.
\end{equation*}
Similar tableaux which are weakly increasing in rows combinatorially
interpret the change-of-basis matrix relating
$\{ h_\lambda \,|\, \lambda \vdash n\}$ and
$\{ m_\lambda \,|\, \lambda \vdash n\}$.

Let $\trspace{n}$ be the $\mathbb Z$-module of
$\zsn$-{\em traces} (equivalently, $\sn$-class functions), 
linear functionals
$\theta: \zsn \rightarrow \mathbb Z$ satisfying
$\theta(gh) = \theta(hg)$ for all $g, h \in \zsn$.
Like the $\mathbb Z$-module $\Lambda_n$,
the trace space 
$\trspace n$ has dimension equal to the number of integer partitions of $n$.  
The Frobenius $\mathbb Z$-module isomorphism (\ref{eq:Frob})
\begin{equation}\label{eq:Frob}
  \begin{aligned}
    \mathrm{Frob}: \trspace n &\rightarrow \Lambda_n \\
    \theta &\mapsto \frac 1{n!} \sum_{w \in \sn} \theta(w) p_{\ctype(w)}
  \end{aligned}
\end{equation}
define bijections
between standard
bases of $\Lambda$, 
and $\trspace n$.
Schur functions correspond to irreducible characters,
\begin{equation*}
  s_\lambda \leftrightarrow \chi^\lambda 
\end{equation*}
while elementary and homogeneous symmetric functions
correspond to induced sign and trivial characters,
\begin{equation*}
  \begin{gathered}
  e_\lambda \leftrightarrow \epsilon^\lambda = \sgn \upparrow^{\sn}_{\mfs \lambda}
  \qquad
  h_\lambda \leftrightarrow \eta^\lambda = \triv \upparrow^{\sn}_{\mfs \lambda}
  \end{gathered}
\end{equation*}
where $\mfs \lambda$ is the Young subgroup of $\sn$ indexed by $\lambda$.
The power sum and monomial bases of $\Lambda_n$ correspond to
bases of $\trspace n$
which are not characters.
We call these the
{\em power sum}
$\{\psi^\lambda \,|\, \lambda \vdash n \}$
and {\em monomial}
$\{ \phi^\lambda \,|\, \lambda \vdash n \}$
traces, respectively.
These are
the bases related to the irreducible character bases
by the same matrices of character evaluations and
inverse Koskta numbers that relate
power sum and monomial symmetric functions to Schur functions,
\begin{equation}\label{eq:mforgotten}
  \begin{alignedat}{2}
    p_\lambda &= \sum_\mu \chi^\mu(\lambda) s_\mu, &\qquad
    \psi^\lambda &= \sum_\mu \chi^\mu(\lambda) \chi^\mu, 
    \\
    m_\lambda &= \sum_\mu K_{\lambda,\mu}^{-1} s_\mu, &\qquad
    \phi^\lambda &= \sum_\mu K_{\lambda,\mu}^{-1} \chi^\mu, 
  \end{alignedat}
\end{equation}
where $\chi^\mu(\lambda) \defeq \chi^\mu(w)$
for any $w \in \sn$ having $\ctype(w) = \lambda$.



\section{Immanants and totally nonnegative polynomials}\label{s:imm}

Each of the inequalities stated in Section~\ref{s:intro} may
be stated in terms of a polynomial in matrix entries. 
In particular, let $x = (x_{i,j})_{i,j \in [n]}$ be a matrix of $n^2$ indeterminates, and for 
$p(x) \in \cx \defeq \mathbb C[x_{i,j}]_{i,j \in [n]}$ and $A = (a_{i,j})$ an $n \times n$ matrix, define $p(A) = p(a_{1,1}, a_{1,2} , \dotsc, a_{n,n})$.
While few of the polynomial inequalities in Section~\ref{s:intro} specifically mention the symmetric group, all of them involve polynomials which are linear combinations of monomials of the form $\{ \permmon xw \,|\, w \in \sn \}$.  Following \cite{LittlewoodTGC}, \cite{StanPos}, we call such polynomials {\em immanants}.  Specifically, given $f: \sn \rightarrow \mathbb C$ define the \emph{$f$-immanant} to be the polynomial
\begin{equation}\label{eq:immdef}
  \imm f(x) \defeq \sum_{w \in \sn} f(w) \permmon xw \in \mathbb C[x].
\end{equation}
The sign character 
($w \mapsto (-1)^{\ell(w)}$)
immanant and trivial character 
($w \mapsto 1$)
immanant are the determinant and permanent,
\begin{equation*}
\det(x) = \sum_{w \in \sn} (-1)^{\ell(w)} \permmon xw,
\qquad
\perm(x) = \sum_{w \in \sn} \permmon xw.
\end{equation*}
Simple formulas for the induced sign and trivial character immanants 
are due to Littlewood--Merris--Watkins~\cite{LittlewoodTGC},
\cite{MerWatIneq},
\begin{equation}\label{eq:lmw}
  \begin{gathered}
    \imm{\epsilon^\lambda}(x) = 
\nTtksp
\sum_{(I_1, \dotsc, I_r)} 
\nTksp
\det(x_{I_1, I_1}) \cdots \det(x_{I_r, I_r}), 
\\
\imm{\eta^\lambda}(x) = 
\nTtksp
\sum_{(I_1, \dotsc, I_r)} 
\nTksp
\perm(x_{I_1, I_1}) \cdots \perm(x_{I_r, I_r}),
\end{gathered}
\end{equation}
where $\lambda = (\lambda_1, \dotsc, \lambda_r) \vdash n$ and
the sums are over all sequences of pairwise disjoint subsets of $[n]$ satisfying $|I_j| = \lambda_j$.
Call such a sequence an {\em ordered set partition} of $[n]$ 
{\em of type $\lambda$.} 

Some current interest in immanants and their connection to TNN matrices was inspired by Lusztig's work with canonical bases of quantum groups. (See, e.g., \cite{LusztigTPCB}.) In particular, one quantum group has an interesting basis whose elements can be described in terms of immanants which evaluate nonnegatively on TNN matrices. 
Call a polynomial $p(x)$ {\em totally nonnegative} (TNN) if $p(A) \geq 0$
whenever $A$ is a totally nonnegative matrix.  There is no known procedure to decide if a given polynomial is TNN.

The formula (\ref{eq:lmw}) makes it obvious that $\imm{\epsilon^\lambda}(x)$
and $\imm{\eta^\lambda}(x)$ are TNN polynomials for each partition
$\lambda \vdash n$.
A stronger result~\cite[Cor.\,3.3]{StemImm} asserts that irreducible character
immanants $\imm{\chi^\lambda}(x)$ are TNN as well.
It is clear that the $\sn$-trace immanants
\begin{equation*}
  \{ \imm{\theta}(x) \,|\,
  \theta \in \trspace n, \imm{\theta}(x) \text{ is TNN } \}
\end{equation*}
form a cone, i.e., are closed under real nonnegative linear combinations.
Stembridge has conjectured~\cite[Conj.\,2.1]{StemConj}
that the extreme rays of this cone are generated by the
monomial trace immanants
\begin{equation}\label{eq:monimm}
  \{ \imm{\phi^\lambda}(x) \,|\, \lambda \vdash n \},
\end{equation}
and has shown~\cite[Prop.\,2.3]{StemConj} that the cone of TNN $\sn$-trace immanants
lies inside of the cone generated by (\ref{eq:monimm}).
\begin{prop}
  Each immanant of the form $\imm{\theta}(x)$ with $\theta \in \trspace n$
  is a totally nonnegative polynomial only if it is equal to
  a nonnegative linear combination of monomial trace immanants.
\end{prop}
Thus it is conjectured that an $\sn$-trace immanant is TNN if and only if it is equal to a nonnegative linear combination of monomial trace immanants. Indeed it is known that some monomial trace immanants generate
extremal
rays of the cone of TNN $\sn$-trace immanants~\cite[Thm.\.10.3]{CHSSkanEKL}.
(See Theorem~\ref{t:monimmtnn}.)

\section{The Temperley-Lieb algebra and 2-colorings}\label{s:tl2c}


Given a complex number $\xi$, we define the 
{\em Temperley-Lieb algebra} $T_n(\xi)$
to be the $\mathbb{C}$-algebra generated by elements
$t_1,\dotsc,t_{n-1}$ subject to the relations
\begin{alignat*}{2}
t_i^2 &= \xi t_i, &\qquad &\text{for } i=1,\dotsc,n-1, \\
t_i t_j t_i &= t_i,   &\qquad &\text{if }  |i-j|=1,\\
t_i t_j &= t_j t_i,   &\qquad &\text{if }  |i-j| \geq 2.
\end{alignat*}
When $\xi = 2$ we have the isomorphism $T_n(2) \cong \csn/(1 + s_1 + s_2 + s_1s_2 + s_2s_1 + s_1s_2s_1)$.
(See e.g.~\cite{FanMon}, \cite[Sec.\,2.1, Sec.\,2.11]{GHJ},
\cite[Sec.\,7]{WestburyTL}.) 
Specificallly, the isomorphism is given by
\begin{equation}\label{eq:sntotn}
\begin{aligned}
\sigma : \csn &\rightarrow T_n(2),\\
s_i &\mapsto t_i - 1. 
\end{aligned}
\end{equation}

Let $\K_n$ be the multiplicative monoid generated
by $t_1,\dotsc,t_{n-1}$ when $\xi = 1$, also called the standard basis of $\tn$.
It is known that $|\K_n|$ is
the $n$th Catalan number $C_n = \tfrac{1}{n+1}\tbinom{2n}{n}$.
Diagrams of the
basis elements of $T_n(\xi)$, made popular by 
Kauffman~\cite[Sec.\,4]{KauffState} are (undirected) graphs with $2n$ vertices and $n$ edges. 
The identity and generators $1, t_1, \dotsc, t_{n-1}$ are represented by
\begin{equation*}
  \raisebox{-6.25mm}{
    {\includegraphics[height=16mm]{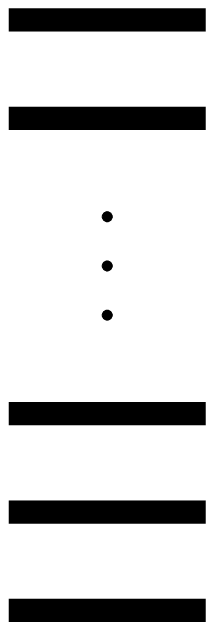}}, ~
    {\includegraphics[height=16mm]{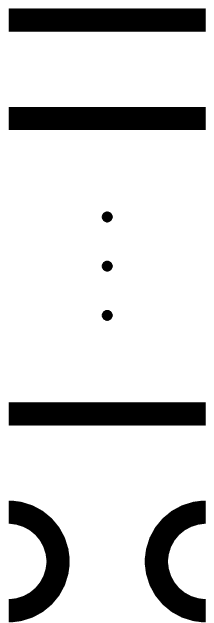}}, ~
    {\includegraphics[height=16mm]{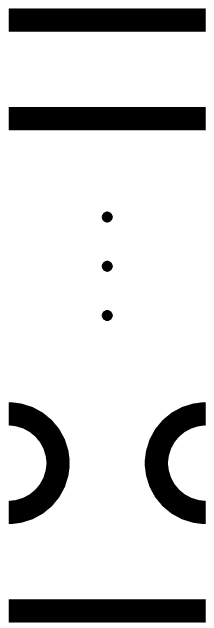}}, $\dotsc$, ~
    {\includegraphics[height=16mm]{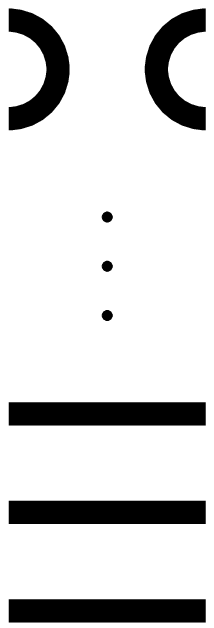}},}
\end{equation*}
and multiplication of these elements corresponds to concatenation
of diagrams, with cycles contributing a factor of $\xi$.
For instance, 
the fourteen basis elements of $T_4(\xi)$ are
\begin{equation*}
\raisebox{-3.1mm}{
  \includegraphics[height=10mm]{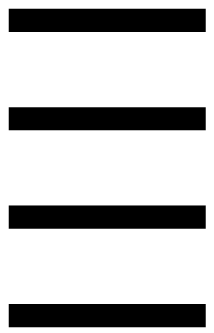}, ~
  \includegraphics[height=10mm]{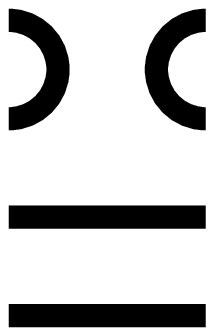}, ~
  \includegraphics[height=10mm]{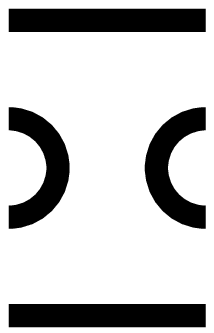}, ~
  \includegraphics[height=10mm]{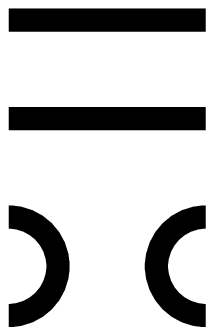}, ~
  \includegraphics[height=10mm]{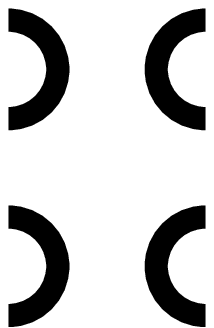}, ~
  \includegraphics[height=10mm]{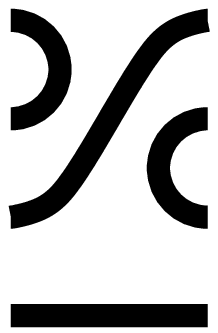}, ~
  \includegraphics[height=10mm]{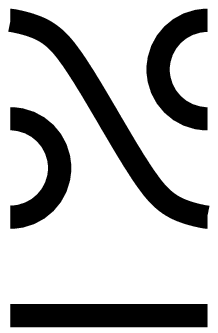}, ~
  \includegraphics[height=10mm]{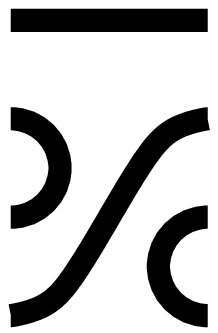}, ~
  \includegraphics[height=10mm]{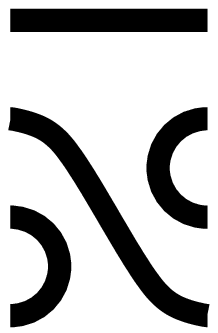}, ~
  \includegraphics[height=10mm]{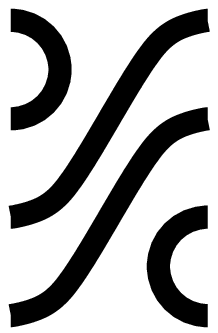}, ~
  \includegraphics[height=10mm]{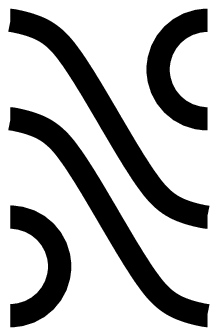}, ~
  \includegraphics[height=10mm]{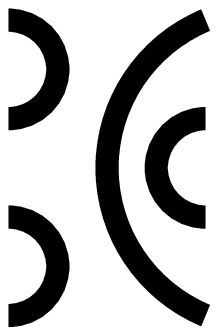}, ~
  \includegraphics[height=10mm]{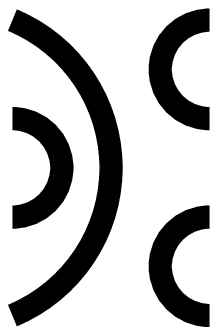}, ~
  \includegraphics[height=10mm]{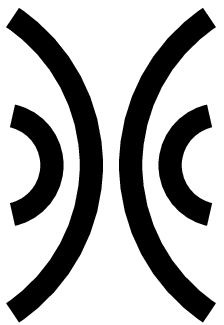},}\\
\end{equation*}
and the equality 
$t_3t_2t_3t_3t_1 = \xi t_1t_3$ in $T_4(\xi)$
is represented by
\begin{equation}\label{eq:kauffex}
  \raisebox{-3.1mm}{
\includegraphics[height=10mm]{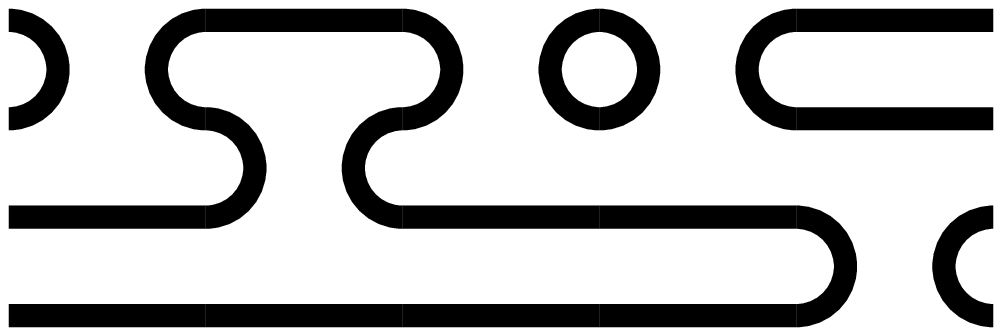}}
 = 
\xi
\raisebox{-3.1mm}{
  \includegraphics[height=10mm]{tl4_t1t3},}
\end{equation}
with the "bubble" becoming the scalar multiple $\xi$.
We will identify each element $\tau \in\K_n$ with its Kauffman diagram, and will 
label vertices $v_1, \dotsc, v_n$, clockwise from the lower left.  We define the height of a vertex by
\begin{equation*}
    \hgt(v_i) = \begin{cases}
      i &\text{if $1 \leq i \leq n$},\\
      2n+1-i &\text{if $n+1 \leq i \leq 2n$}.
      \end{cases}
\end{equation*}
For instance, (\ref{eq:tauhat}) shows the element $t_7t_6t_8t_5t_7t_4t_6t_5t_2 \in \K_9$ 
on the left, with vertex labels.
Vertices have heights $1,\dotsc,9$, from bottom to top.
It is easy to see that for all $\tau \in\K_n$, each edge $(v_i,v_j)$ satisfies
\begin{equation}\label{eq:heightdiff}
\hgt(v_i) - \hgt(v_j) = \begin{cases}
 1 ~(\mathrm{mod}~ 2) &\text{if $i,j \leq n$ or $i,j \geq n+1$},\\
 0 ~(\mathrm{mod}~ 2) &\text{otherwise}.
\end{cases}
\end{equation}

Define $\hat \tau$ to be the graph obtained from $\tau$ by adding edges $(i, 2n+1-i)$ for all $i$ (even if such an edge already exists).  Since each vertex in $\hat \tau$ has degree $2$, it is clear that this graph is a disjoint union of cycles.  For example, corresponding to the element $\tau$ on the left of (\ref{eq:tauhat}) we have $\hat \tau$ to its right, and the decomposition of this graph into four disjoint cycles.
\begin{equation}\label{eq:tauhat}
  \begin{gathered}
 \raisebox{-3.1mm}{
{\includegraphics[height=60mm]{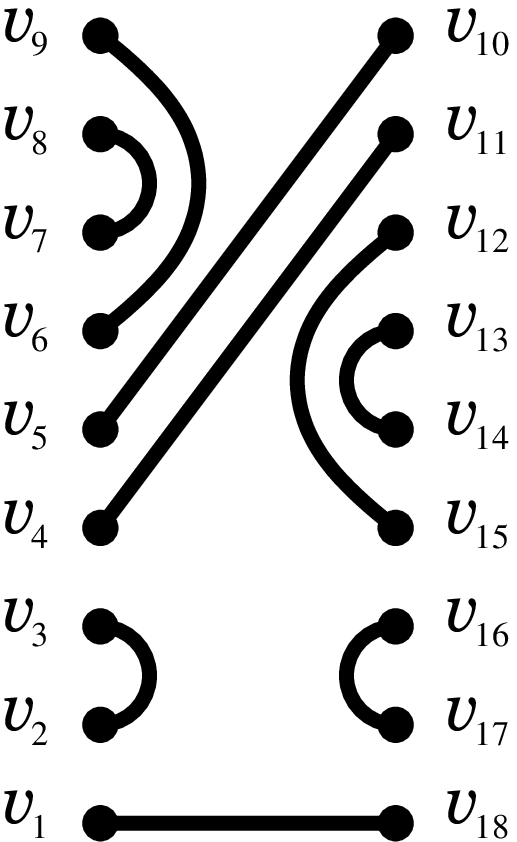}}} \\
\phantom \sum \tau \phantom \sum
\end{gathered}
\qquad\qquad
\begin{gathered}
\raisebox{-3.1mm}{
\includegraphics[height=60mm]{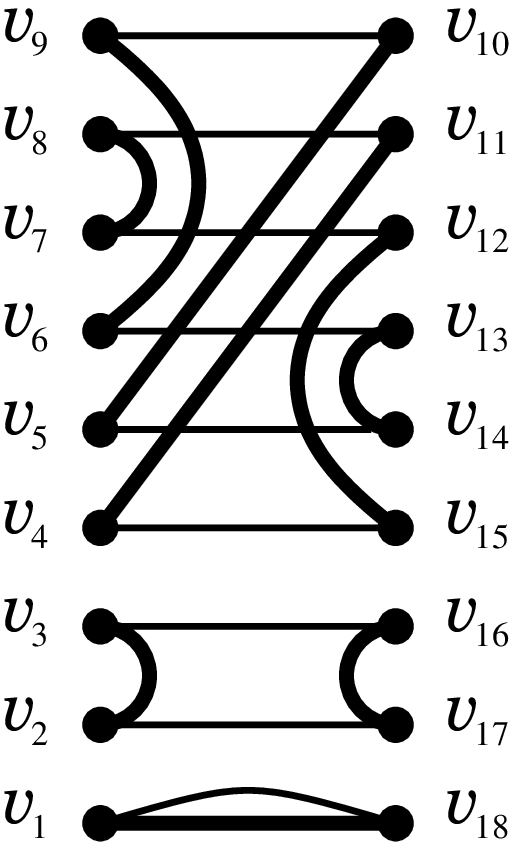}}\\
\phantom\sum \hat \tau \phantom \sum\end{gathered}
\qquad \qquad
\begin{gathered}
\raisebox{-3.1mm}{
\includegraphics[height=60mm]{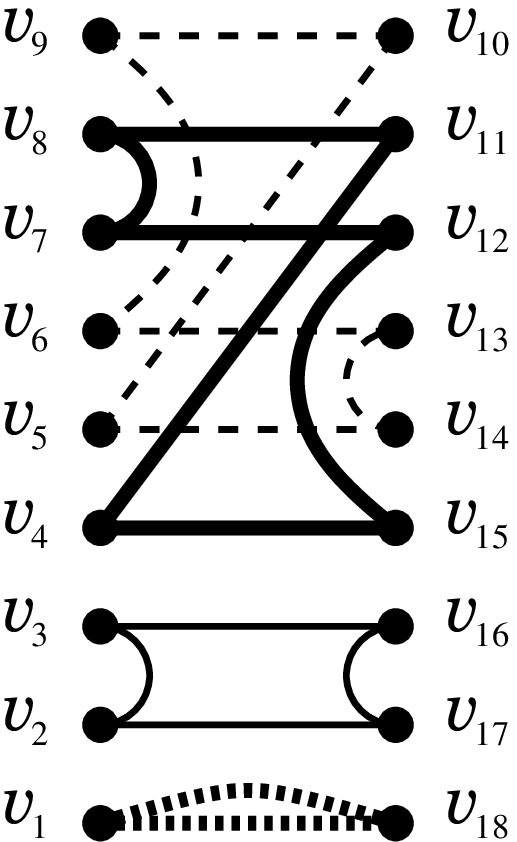}}\ .\\
\phantom\sum
\text{cycles of } \hat\tau\phantom\sum
\end{gathered}
\end{equation}

The Temperley-Lieb algebra $T_n(2)$ sometimes arises in the $2$-coloring
of combinatorial objects.
Define a {\em principal coloring} of $\tau \in \K_n$ to be a map $\kappa: \mathrm{vertices}(\tau) \rightarrow \{ \mathrm{black}, \mathrm{white} \}$ which is a {\em proper} coloring of $\hat \tau$, i.e., 
\begin{equation*}
\begin{alignedat}{2} 
\text{color}(v_i) &\neq \text{color}(v_{2n+1-i}) &\qquad  &\text{for $i=1,\dotsc,n$},\\
\text{color}(v_i) &\neq \text{color}(v_j) &\qquad &\text{if $v_i$ and $v_j$ are adjacent in $\tau$}.  
\end{alignedat}
\end{equation*}
Let $(\tau, \kappa)$ denote the graph $\tau$ with its vertices colored by $\kappa$.

Principal colorings of $\tau$ are closely related to cycles in $\hat\tau$.  It is clear that colors must alternate along any one cycle of $\hat\tau$.  It is also true that vertex colors alternate as one views the vertices in clockwise order, ignoring the edges of that cycle.
For example, consider a $2$-coloring of the cycle $(v_4, v_{11}, v_8, v_7, v_{12}, v_{15})$ of $\hat\tau$ in (\ref{eq:tauhat}),
\begin{equation*}
    \raisebox{-3.1mm}{
\includegraphics[height=30mm]{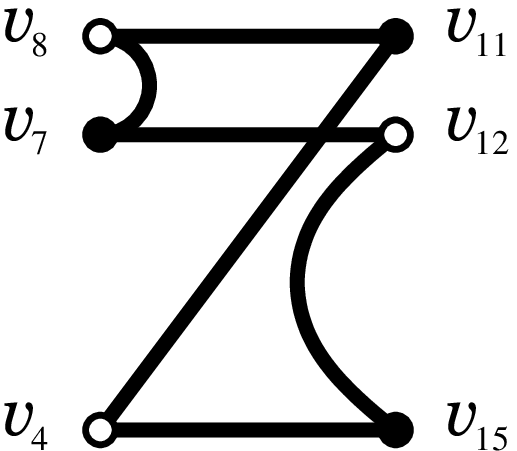}}.
\end{equation*}
\begin{prop}
If $\hat\tau$ is a single cycle, then there are two principal colorings of $\tau$. In each, vertices of odd index and of even index have opposite colors. 
\end{prop}
\begin{proof}
This is clear if $n=1$. Assume that $n\geq2$ and suppose we have a principal coloring of $\tau$. It suffices to show that color($v_i$)$~\neq~$color($v_{i+1}$) for $i=1,\dotsc,n-1.$ Consider a subpath of the cycle of $\hat\tau$ from $v_i$ to $v_{i+1}$. Since $v_{i}$, $v_{i+1}$ belong to the left column, there are an even number of edges on the path that switch columns. By (\ref{eq:heightdiff}), for each such edge $(v_{j},v_{k})$ we have that $\hgt(v_j)-\hgt(v_k)$ is even. Also by (\ref{eq:heightdiff}), since $\hgt(v_{i+1})-\hgt(v_i)=1$ there are an odd number of edges that do not switch columns. Thus the path consists of an odd number of edges and color($v_i$)$~\neq~$color($v_{i+1}$).
\end{proof}

It is easy to see that if $\kappa$ is a principal coloring of $\tau \in \K_n$ and if $\hat \tau$ is a single cycle, then we have
\begin{equation}\label{eq:lrdiff}
    | \# (\text{white vertices on left of } (\tau,\kappa)) - 
    \# (\text{white vertices on right of } (\tau,\kappa)) | = 
    \begin{cases}
      0 & \text{if $n$ even},\\
      1 & \text{if $n$ odd}.
      \end{cases}
      \end{equation}
In this situation we call $(\tau,\kappa)$ {\em balanced} if $n$ is even, and
{\em unbalanced} otherwise.
More specifically, we call $(\tau,\kappa)$ {\em left-unbalanced} ({\em right-unbalanced}) if it has more white vertices on the left (right).

Now consider $\tau \in \K_n$
with $\hat \tau$ a disjoint union of cycles $C_1,\dotsc, C_d$, and $\kappa$ a principal coloring of $\tau$.
Define
\begin{equation}\label{eq:alphabeta}
\begin{aligned}
\alpha &= \alpha(\tau,\kappa) \defeq \# \{ i \,|\, (\tau_{C_i}, \kappa) \text{ right unbalanced} \},\\   \beta &= \beta(\tau,\kappa) \defeq \# \{ i \,|\, (\tau_{C_i}, \kappa) \text{ left unbalanced} \}.
\end{aligned}
\end{equation}
For example, recall $\tau$ and $\hat\tau$ from (\ref{eq:tauhat}).  The proper $2$-coloring $\kappa$ of $\hat\tau$,
\begin{equation*}
\raisebox{-3.1mm}{
\includegraphics[height=60mm]{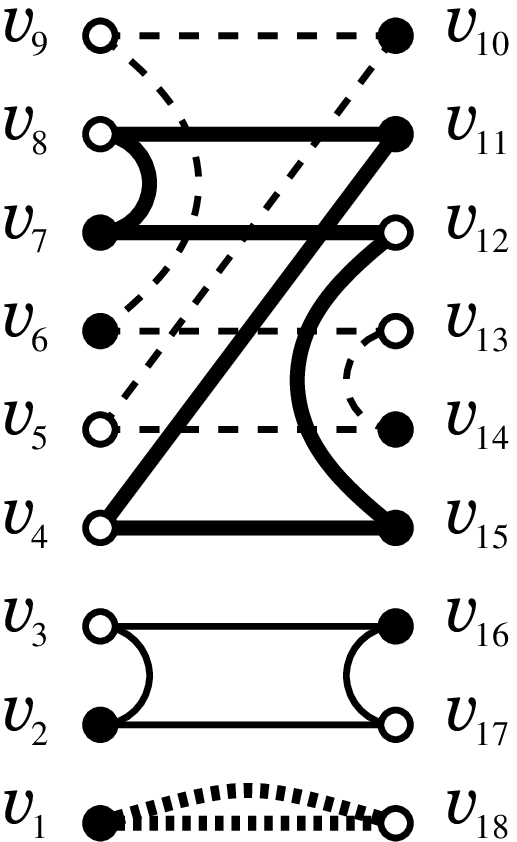}},
\end{equation*}
  corresponds to a principal coloring of $\tau$ which satisfies $\alpha(\tau,\kappa) = 1$, $\beta(\tau,\kappa) = 2$, and $d=4$. Also note that there is one balanced cycle.

It is easy to characterize the colorings $\kappa$ of a given Temperley-Lieb basis element $\tau$ for which the numbers $\alpha$, $\beta$ (\ref{eq:alphabeta}) are constant.
\begin{lem}
Let $(\tau, \kappa)$, $(\tau,\kappa')$ be principal colorings with $j$ white vertices on the left, for some $j$, $0 \leq j \leq \lfloor \tfrac n2 \rfloor$.
Then we have
\begin{equation*}
    \alpha(\tau,\kappa) = \alpha(\tau,\kappa'),
    \qquad
    \beta(\tau,\kappa) = \beta(\tau,\kappa').
\end{equation*}
\end{lem}
\begin{proof}
  It is easy to see that the number of cycles of $\hat \tau$ of cardinality $2~ (\mathrm{mod}~4)$ is
  \begin{equation}\label{eq:absum}
      \alpha(\tau,\kappa) + \beta(\tau,\kappa) = 
      \alpha(\tau,\kappa') + \beta(\tau,\kappa').
  \end{equation}
  On the other hand, we have by assumption that the number of white vertices on the right of $(\tau,\kappa)$ (or of $(\tau,\kappa')$) minus the number of white vertices on the left is \begin{equation}\label{eq:njdiff}
  (n-j)-j = n - 2j.
  \end{equation}
  Since each balanced subgraph $(\tau_{C_i},\kappa)$ (or $(\tau_{C_i},\kappa')$) contributes the same number of white vertices to both sides 
  and each unbalanced subgraph contributes one more white vertex to one side than to the other (\ref{eq:lrdiff}),
  we see that (\ref{eq:njdiff}) equals
  \begin{equation}\label{eq:abdiff}
  \alpha(\tau,\kappa) - \beta(\tau,\kappa) = \alpha(\tau,\kappa') - \beta(\tau,\kappa').
  \end{equation}
  Combining (\ref{eq:absum}) and (\ref{eq:abdiff}), we have the desired equalities.
\end{proof}
\noindent
By this lemma, we may write
\begin{equation}\label{eq:newalpha}
\alpha(\tau,j) \defeq \alpha(\tau,\kappa)
\qquad 
(\beta(\tau,j) \defeq \beta(\tau,\kappa))
\end{equation}
if there exists a principal coloring of $\tau$ in which $j$ vertices on the left are white.

Just as $T_n(2)$ is related to $2$-coloring, it is related to total
nonnegativity of polynomials in the subspace
\begin{equation}\label{eq:tlspace}
  \spn_{\mathbb R} \{ \det(x_{I,I}) \det(x_{I^c,I^c}) \,|\, I \subseteq [n] \}
  \end{equation}
of immanants. We define an immanant $\imm{\tau}(x)$ for each
$\tau \in \K_n$ in terms of the function
\begin{equation}\label{eq:ftau}
\begin{aligned}
  f_\tau: \csn &\rightarrow \mathbb{R}\\
  w &\mapsto \text{ coefficient of $\tau$ in } \sigma(w),
    \end{aligned}
\end{equation}
(extended linearly).
To economize notation, we will write $\imm{\tau}$  instead of $\imm{f_\tau}$,
\begin{equation*}
\imm{\tau}(x) 
= \sum_{w \in \sn} f_\tau(w)x_{1,w_1} \cdots x_{n,w_n}.
\end{equation*}
For example, consider the case $n=3$ and $\tau = t_1 \in \K_3$.
Extracting the coefficients 
of $t_1$ in the expressions
\begin{equation*}
  \begin{gathered}
    \sigma(e) = 1, \qquad \sigma(s_1) = t_1 - 1, \qquad \sigma(s_2) = t_2 - 1,\\
    \sigma(s_1s_2) = (t_1 - 1)(t_2 - 1) = t_1t_2 - t_1 - t_2 + 1, \\
    \sigma(s_2s_1) = (t_2 - 1)(t_1 - 1) = t_2t_1 - t_1 - t_2 + 1, \\
    \sigma(s_1s_2s_1) = (t_1-1)(t_2-1)(t_1-1) = t_1 + t_2 - t_1t_2 - t_2t_1 - 1,
  \end{gathered}
\end{equation*}
(where we have used $t_1t_2t_1 = t_1$ and $t_1^2 = 2t_1$
to obtain the last expression), we have
$f_{t_1}(e) = 0$,
$f_{t_1}(s_1) = 1$,
$f_{t_1}(s_2) = 0$,
$f_{t_1}(s_1s_2) = -1$,
$f_{t_1}(s_2s_1) = -1$,
$f_{t_1}(s_1s_2s_1) = 1$,
and
\begin{equation*}
  \imm{t_1}(x) = x_{1,2}x_{2,1}x_{3,3}
  - x_{1,3}x_{2,1}x_{3,2}
  - x_{1,2}x_{2,3}x_{3,1}
  + x_{1,3}x_{2,2}x_{3,1}.
  \end{equation*}
Note that in the special case $\tau = 1$, the function $f_\tau$ 
maps a permutation $w$ to $(-1)^{\inv(w)}$.  Thus 
the determinant is a Temperley-Lieb immanant,
\begin{equation*}
\det(x) = \imm{1}(x).
\end{equation*}

It was shown in \cite{RSkanTLImmp} that Temperley-Lieb immanants
are a basis of the space (\ref{eq:tlspace}), and that they are TNN.
Furthermore, they are the extreme rays of the cone of TNN immanants in this
space~\cite[Thm.\,10.3]{RSkanTLImmp}.
\begin{prop}
  Each immanant of the form
  \begin{equation}\label{eq:sumofprodsof2minors}
    \imm f(x) = \sumsb{I,J \subseteq [n]\\ |I|=|J|} c_{I,J}
    \det(x_{I,J}) \det(x_{\ol I, \ol J})
  \end{equation}
  is a totally nonnegative polynomial if and only if it is equal to
  a nonnegative linear combination of Temperley-Lieb immanants.
\end{prop}
In fact each complementary product of minors is a $0$-$1$ linear
combination of Temperley-Lieb immanants~\cite[Prop.\,4.4]{RSkanTLImmp}.
\begin{thm}\label{t:oneprod}
For $I \subseteq [n]$ we have 
\begin{equation}
\det(x_{I,I})\det(x_{I^c,I^c})=\sum_{\tau \in \K_n}b_{\tau}\imm{\tau}(x),    
\end{equation}
where 
\begin{equation}
    b_{\tau}=\begin{cases}
      1 & \text{if there is a principal coloring of $\tau$ with $\{v_i\,|\,i \in I\}$ white, $\{v_i\,|\,i \in [n]\ssm I\}$ black}, \\
      0&\text{otherwise.}
    \end{cases}
\end{equation}
\end{thm}
By (\ref{eq:lmw}) we have that for each two-part partition
$\lambda = (n-j,j)$ of $n$, the corresponding induced sign character immanant
belongs to (\ref{eq:tlspace}).
Furthermore, we have the following explicit expansion of these in terms of the
Temperley-Lieb immanant basis.
\begin{thm}\label{t:epsilontau}
For $j = 0,\dotsc,\lfloor \tfrac n2 \rfloor - 1$, we have
\begin{equation*}
    \imm{\epsilon^{n-j,j}}(x) = \sum_{\tau \in \K_n} d_{j,\tau} \imm{\tau}(x),
    \end{equation*}
    where $d_{j,\tau}$
    is the number of principal colorings of $\tau$ having $j$ white vertices on the left. Explicitly, assuming such a coloring exists, this is $2^{d-\alpha-\beta}\binom{\alpha+\beta}{\alpha}$, where
    $d=$ the number of cycles of $\hat\tau$, and
    $\alpha=\alpha(\tau,j)$, $\beta=\beta(\tau,j)$
    are defined as in (\ref{eq:newalpha}).
    \begin{proof}
    The combinatorial description follows immediately from Theorem~\ref{t:oneprod}. Now suppose $\hat\tau$ consists of cycles $C_1,\dotsc,C_d$ and consider the proper 2-coloring of these cycles that combine to form a principal coloring of $\tau$ having $j$ white vertices on the left. For each of the $d-\alpha-\beta$ balanced induced subgraphs $\tau_{C_i}$, both of the two possible colorings contribute $\frac{|C_i|}{2}$ white vertices to the left column of $\tau$. There are $2^{d-\alpha-\beta}$ colorings of the corresponding vertices. Besides these, each unbalanced subgraph $\tau_{C_i}$ must be colored so that it contributes more white vertices to the left or to the right. We choose $\alpha$ of these $\alpha+\beta$ unbalanced subgraphs to be right unbalanced in $\binom{\alpha+\beta}{\alpha}$ ways.
    \end{proof}
\end{thm}

Combining (\ref{eq:etom}) and (\ref{eq:lmw}), we see that
monomial immanants $\imm{\phi^\mu}(x)$ indexed by
partitions of the form $\mu = 2^c 1^d \vdash n$ belong to the space
(\ref{eq:tlspace}) as well.
To expand these
in the Temperley--Lieb immanant basis,
we define for each $\mu = 2^c 1^d \vdash n$ the set $P(\mu)$
of all $\tau \in \K_n$ such that there exists a principal coloring of $\tau$
with $c+d$ white vertices on the left and
no principal coloring of $\tau$ with $c+d+1$ white vertices on the left.
\begin{thm}\label{t:monimmtnn}
  For $\mu = 2^c 1^d \vdash n$, we have
  that $\imm{\phi^\mu}(x)$ is a totally nonnegative polynomial. In particular we  have 
\begin{equation*}
    \imm{\phi^\mu}(x) = \sum_{\tau \in P(\mu)} b_{\mu,\tau} \imm\tau(x), 
\end{equation*}
where $b_{\mu,\tau}=2^{\# \text{cycles of } \hat\tau \text{ of cardinality } 0~(\text{mod }4)}$.
\end{thm}
\begin{proof}
Let $t_{i_1}\cdots t_{i_\ell}$ be an expression for $\tau$ which is as short as possible, and let $G$ be the wiring diagram of $s_{i_1}\cdots s_{i_\ell}$. By \cite[Thm.\,10.3]{CHSSkanEKL}, the coefficient $b_{\mu,\tau}$ is the number of path families covering $G$ which satisfy 
\begin{enumerate}
    \item $c+d$ pairwise nonintersecting paths are colored white, 
    \item $c$ pairwise nonintersecting paths are colored black, 
    \end{enumerate} 
assuming no path family covering $G$ can be colored
so that $c+d+1$ nonintersecting paths are colored white.
It is straightforward to show \cite[Prop.\,2.1]{SkanIneq} that
such path families correspond bijectively to principal colorings
of $\tau$ with $c+d$ white and $c$ black vertices on the left,
assuming that no principal coloring of $\tau$ has $c+d+1$
white vertices on the left.
    
    Let $(\tau,\kappa)$ be a principal coloring with $\alpha(\tau,\kappa)>0$. Let $(\tau,\kappa')$ be a coloring obtained from $(\tau,\kappa)$ by switching the color of each vertex in one right unbalanced subgraph of $(\tau,\kappa)$. In $(\tau,\kappa')$ the total number of white vertices on the left is one bigger than in $(\tau,\kappa)$. Thus, for principal colorings described above $\alpha(\tau,\kappa)=0$ and by Theorem~\ref{t:epsilontau} we have $b_{\mu,\tau}=2^{d- \beta}$. Since $\alpha=0$, we have that $d-\beta=\# \text{cycles of } \hat\tau \text{ of cardinality } 0~(\text{mod }4)$, as desired.
\end{proof}

\section{Main results}\label{s:mn}

Fischer's inequalities (\ref{eq:fischer}) naturally lead to the questions of how products
\begin{equation}\label{eq:prodminors}
\det(A_{I,I})\det(A_{I^c,I^c})
\end{equation}
of complementary pairs of minors compare to one another, and of whether a greater cardinality difference $|I^c| - |I|$ tends to make the product (\ref{eq:prodminors}) greater or smaller.
This second question led Barrett and Johnson~\cite{BJMajor} to consider the average value of such products when cardinalities are fixed,
\begin{equation}\label{eq:avg2}
    \frac 1{\binom nk} \sumsb{I \subseteq [n]\\|I| = k}
    \det(A_{I,I})\det(A_{I^c,I^c}).
    \end{equation}
    They found that for PSD matrices, a smaller cardinality difference makes the average product of minors greater~\cite[Thm.\,1]{BJMajor}.
    We give two proofs that the same is true for TNN matrices.

\begin{thm}\label{t:fischer} For all TNN matrices A and for $k = 0,\dotsc, \lfloor \tfrac n2 \rfloor - 1,$ we have
\begin{equation}\label{eq:fisher}
    \frac{\displaystyle{\sum}_{|I|=k} \det(A_{I,I}) \det(A_{I^c,I^c})}{\displaystyle{\binom nk}}
    \leq
    \frac{\displaystyle{\sum}_{|I|=k+1} \det(A_{I,I}) \det(A_{I^c,I^c})}{\displaystyle{\binom n{k+1}}}.
    \end{equation}
\end{thm}
\begin{proof}[First proof]
By (\ref{eq:lmw}) it is equivalent to show that the polynomial
\begin{equation}\label{eq:ediff}
    \frac{\imm{\epsilon^{n-k-1,k+1}}(x)}{\displaystyle{\binom n{k+1}}} - 
    \frac{\imm{\epsilon^{n-k,k}}(x)}{\displaystyle{\binom nk}}
\end{equation}
is totally nonnegative.  By Theorem~\ref{t:epsilontau}, this difference belongs to the span of the Temperley-Lieb immanants.  In particular, we may multiply the difference by $n!/(k!(n-k-1)!)$ to obtain
\begin{equation}\label{eq:fishertlnofrac}
    (k+1)\imm{\epsilon^{n-k-1,k+1}}(x) - 
    (n-k)\imm{\epsilon^{n-k,k}}(x) 
    = \sum_{\tau \in \K_n} c_\tau \imm \tau (x),
\end{equation}
where 
\begin{equation}\label{eq:ddc}
c_\tau = (k+1)d_{k+1,\tau} - (n-k)d_{k,\tau}.
\end{equation}
and $d_{k+1,\tau}$, $d_{k,\tau}$ are defined in terms of proper colorings of $\tau$ as in Theorem~\ref{t:epsilontau} 

We claim that $c_\tau \geq 0$.
Suppose we have a proper coloring of $\tau$ in which $k$ vertices on the left are white. Let $\alpha = \alpha(\tau,k)$ and $\beta = \beta(\tau,k)$ be the numbers of right-unbalanced and left-unbalanced subgraphs of $\tau$ respectively, as in \ref{eq:newalpha}.

Let $d$ be the number of cycles defined as in Theorem~\ref{t:epsilontau}. Then by Theorem~\ref{t:epsilontau} we have that (\ref{eq:ddc}) equals
\begin{equation}\label{eq:abgamma}
    \begin{aligned}
    (k+1)d_{k+1,\tau} - (n-k)d_{k,\tau} &=
    (k+1)2^{d-\alpha-\beta}\binom{\alpha+\beta}{\alpha-1}
    -(n-k)2^{d-\alpha-\beta}\binom{\alpha+\beta}{\alpha} \\
    &= 2^{d-\alpha-\beta}\left( (k+1)\binom{\alpha+\beta}{\alpha-1} - (n-k) \binom{\alpha+\beta}{\alpha} \right)\\
    &= 2^{d-\alpha-\beta}\frac{(\alpha+\beta)!}{(\alpha-1)!\beta!} \left( \frac{k+1}{\beta+1} - \frac{n-k}{\alpha} \right).
    \end{aligned}
\end{equation}
Then substituting $\alpha = n - 2k + \beta$ (see formula \ref{eq:njdiff}), the final difference of fractions (\ref{eq:abgamma}) becomes
\begin{equation*}
    \frac{(k+1)(n-2k+\beta) - (n-k)(\beta+1)} {(\beta+1)(n-2k+\beta)}.
\end{equation*}
This denominator is clearly positive.  The numerator is $(n-2k - 1)(k-\beta)$,
which is nonnegative because of the bounds on $k$ and the definition of $\beta$.
\end{proof}


\begin{proof}[Second proof]
Again we show that the polynomial (\ref{eq:ediff}) is TNN.
  Multiplying the polynomial by
  $\tfrac{n!}{k!(n-k-1)!}$ and expanding it in the monomial immanant
  basis of the trace immanant space, we have
  \begin{equation}\label{eq:fisherpoly2}
    (n-k)\imm{\epsilon^{(k,n-k)}}(x) -
    (k+1)\imm{\epsilon^{(k+1,n-k-1)}}(x) = \sum_{\mu \vdash n} c_{\mu}\imm{\phi^\mu}(x)
  \end{equation}
  where the integers $\{ c_{\mu} \,|\, \mu \vdash n \}$ satisfy
  \begin{equation}\label{eq:efisher}
    (n-k)e_{(k,n-k)} - (k+1)e_{(k+1,n-k-1)} = \sum_{\mu \vdash n} c_{\mu}m_\mu.
  \end{equation}

  We claim that $c_{\mu} \geq 0$ for all relevant partitions $\mu$.
  To see this, consider the special case of (\ref{eq:etom}) 
  for two-part partitions,
  \begin{equation*}
    e_{(k,n-k)} = \sum_{a=0}^{n-k} M_{(k,n-k),2^a1^{n-2a}} m_{2^a1^{n-2a}}.
  \end{equation*}
  It follows that each partition $\mu$ appearing with nonzero coefficient in (\ref{eq:fisherpoly2})
  has the form $\mu = 2^a1^{n-2a}$.
  
  Since each column-strict Young tableaux of shape $(k,n-k)$ and content
  $2^a 1^{n-2a}$ contains the letters $1, 1, 2, 2, \dotsc, a, a$ in rows
  $1, \dotsc, a$, it is completely determined by the subset of
  the
  letters $\{a+1, \dotsc, n-a\}$
  appearing in rows $a+1,\dotsc,n-k$ of column $2$.
  Thus there are $\tbinom{n-2a}{n-k-a} = \tbinom{n-2a}{k-a}$ such tableaux.
  It follows that in the right-hand side of (\ref{eq:efisher}),
  the coefficient of $m_{2^a1^{n-2a}}$ is
  \begin{equation*}
    c_{2^a1^{n-2a}} = \begin{cases}
      (n-k)\binom{n-2a}{k-a} 
      &\text{if $a = n-k$,}\\
      (n-k)\binom{n-2a}{k-a} - (k+1)\binom{n-2a}{k-a+1}
      &\text{if $0 \leq a \leq n-k-1$.}
    \end{cases}
  \end{equation*}
This last expression is equal to
  \begin{multline*}
    \frac{(n-2a)!}{(k-a)!(n-k-a-1)!}
      \bigg( \frac{n-k}{n-k-a} - \frac{k+1}{k-a+1} \bigg)\\ 
      = \frac{(n-2a)!}{(k-a)!(n-k-a-1)!}
      \bigg( \frac{2k+1-n}{(n-k-a)(k-a+1)} \bigg),
  \end{multline*}
  which is positive, since
  $a < n-k \leq \lfloor \tfrac n2 \rfloor \leq k \leq n$.

  Now by Theorem~\ref{t:monimmtnn}
  we have that
  (\ref{eq:fisherpoly2})
  is a totally nonnegative polynomial, and that
  (\ref{eq:ediff}) is as well.
\end{proof}
  
Consideration of the average value of products of two complentary minors (\ref{eq:avg2}) naturally led Barrett and Johnson to consider average values of products of arbitrarily many minors having fixed cardinality sequences $\lambda = (\lambda_1,\dotsc,\lambda_r) \vdash n$, 
\begin{equation}\label{eq:avgr}
\frac{1}{\binom{n}{\lambda_1,\dotsc,\lambda_r}}
\sum_{(I_1,\dotsc,I_r)}\det(A_{I_1,I_1}) \cdots \det(A_{I_r,I_r}),
\end{equation}
where the sum is over ordered set partitions of $[n]$ of type $\lambda$ and the multinomial coefficient $\binom{n}{\lambda_1,\dotsc,\lambda_r} = \tfrac{n!}{\lambda_1! \cdots \lambda_r!}$ is the number of such ordered set partitions.  They found that for PSD matrices, partitions which appear lower in the majorization order lead to greater average products~\cite[Thm.\,2]{BJMajor}.  We show that the same is true for TNN matrices.  
  
\begin{thm}\label{t:main}
Fix $n > 0$ and partitions $\lambda = (\lambda_1,\dotsc,\lambda_r) \vdash n$,
$\mu = (\mu_1 \dotsc, \mu_s) \vdash n$ with $\lambda \preceq \mu$.
For all TNN matrices we have
  \begin{equation}\label{eq:majorfischerpoly}
\lambda_1!\cdots\lambda_r!
\nTksp\sum_{(I_1,\dotsc,I_r)}\nTksp
\det(A_{I_1,I_1}) \cdots \det(A_{I_r,I_r}) ~\geq~
  \mu_1!\cdots\mu_s!
  \nTksp\sum_{(J_1,\dotsc,J_s)}\nTksp
  \det(A_{J_1,J_1}) \cdots \det(A_{J_s,J_s}),
  \end{equation}
\end{thm}
\begin{proof}
  By (\ref{eq:movebox}) it suffices to consider $\lambda$, $\mu$ having equal parts except
  \begin{equation*}
  \mu_i = \lambda_i + 1, \qquad \mu_j = \lambda_j - 1
  \end{equation*}
  for some $i < j$. (Thus we may assume $s = r$ and will allow $\mu_s = 0$.)
  Let $\nu = (\nu_1,\dotsc,\nu_{r-2})$ be the partition of $n - \lambda_i - \lambda_j$ consisting of all other parts.
  
As in the proofs of Theorem~\ref{t:fischer}, we observe that each sum of products of minors is an induced sign character immanant (\ref{eq:lmw}). Thus the inequality (\ref{eq:majorfischerpoly}) is equivalent to the total nonnegativity of the polynomial
      \begin{equation}\label{eq:thediff}
      \lambda_1! \cdots \lambda_r!\; \imm{\epsilon^\lambda}(x) - 
      \mu_1! \cdots \mu_r!\;
      \imm{\epsilon^\mu}(x).
      \end{equation}
      Dividing this by $\nu_1! \cdots \nu_{r-2}! \lambda_i! \mu_j!$ and collecting terms, we obtain
  \begin{equation*}
  \begin{gathered}
      \lambda_j 
      \ntksp \sumsb{J \subseteq [n]\\|J| = |\nu|} \ntksp \imm{\epsilon^\nu}(x_{J,J}) \imm{\epsilon^{(\lambda_i,\lambda_j)}}(x_{J^c,J^c}) - 
      \mu_i 
      \ntksp \sumsb{J \subseteq [n]\\|J| = |\nu|}\ntksp \imm{\epsilon^\nu}(x_{J,J}) \imm{\epsilon^{(\lambda_i+1,\lambda_j-1)}}(x_{J^c,J^c}) \\
      = \ntksp \sumsb{J \subseteq [n]\\|J| = |\nu|} \imm{\epsilon^\nu}(x_{J,J})
      \left(\lambda_j \imm{\epsilon^{(\lambda_i,\lambda_j)}}(x_{J^c,J^c})
      -
      \mu_i 
      \imm{\epsilon^{(\lambda_i+1,\lambda_j-1)}}(x_{J^c,J^c}) \right).
      \end{gathered}
  \end{equation*}
  By Theorem~\ref{t:fischer}, or more precisely (\ref{eq:fishertlnofrac}),
  this polynomial is TNN and so is (\ref{eq:thediff}).
\end{proof}

It would be interesting to extend the Barrett--Johnson inequalities (\ref{eq:bj}) 
to all HPSD matrices, and to state a permanental analog as 
originally suggested in \cite{BJMajor}.
Using (\ref{eq:lmw})
we may state these problems 
as follows.
\begin{prob}\label{p:bj}
Characterize the pairs $(\lambda,\mu)$
of partitions of $n$ for which we have
\begin{enumerate}
    \item $\lambda_1! \cdots \lambda_r! \;\imm{\epsilon^\lambda}(A) \leq 
    \mu_1! \cdots \mu_r! \;\imm{\epsilon^\mu}(A)$ for all $A$ HPSD,
    \item $\lambda_1! \cdots \lambda_r! \;\imm{\eta^\lambda}(A) \geq
    \mu_1! \cdots \mu_r! \;\imm{\eta^\mu}(A)$ for all $A$ HPSD or PSD,
    \item $\lambda_1! \cdots \lambda_r! \;\imm{\eta^\lambda}(A) \geq
    \mu_1! \cdots \mu_r! \;\imm{\eta^\mu}(A)$ for all $A$ TNN.
\end{enumerate}
\end{prob}

\section{Acknowledgements} The authors are grateful to Shaun Fallat
and Misha Gekhtman for
helpful conversations.  

\bibliography{skan}
\end{document}